\documentclass[12pt]{article}
\usepackage{amssymb}
\usepackage{amsthm}
\usepackage{amsmath}
\usepackage{graphicx}
\usepackage{enumerate}

\DeclareMathOperator{\Ded}{Ded}
\DeclareMathOperator{\BDed}{\mathbf{Ded}}

\newtheorem{theorem}{Theorem}
\newtheorem{definition}[theorem]{Definition}
\newtheorem{lemma}[theorem]{Lemma}
\newtheorem{proposition}[theorem]{Proposition}

\newtheorem{example}[theorem]{Example}

\title{Properties of the connective implication in effect algebras}
\author{Ivan~Chajda and Helmut~L\"anger}
\date{}
\begin{document}
\footnotetext[1]{Support of the research by \"OAD, project CZ~02/2019, and support of the research of the first author by IGA, project P\v rF~2019~015, is gratefully acknowledged.}
\maketitle
\begin{abstract}
Effect algebras form a formal algebraic description of the structure of the so-called effects in a Hilbert space which serves as an event-state space for effects in quantum mechanics. This is why effect algebras are considered as logics of quantum mechanics, more precisely as an algebraic semantics of these logics. Because every productive logic is is equipped with a connective implication, we introduce here such a concept and demonstrate its properties. In particular, we show that this implication is connected with conjunction via a certain ``unsharp'' residuation which is formulated on the basis of a strict unsharp residuated poset. Though this structure is rather complicated, it can be converted back into an effect algebra and hence it is sound. Further, we study the Modus Ponens rule for this implication by means of so-called deductive systems and finally we study the contraposition law. 
\end{abstract}
 


Effect algebras were introduced by D.~J.~Foulis and M.~K.~Bennett (\cite{FB}) in order to axiomatize so-called quantum logic effects in a Hilbert space which serve as an event-state space for quantum mechanics. Hence, effect algebras are considered as a logic of quantum mechanics, see e.g.\ \cite{DV} for details. Effect algebras are partial algebras with one partial binary operation, one unary (induced) operation and two constants. When an effect algebra is considered as a formal propositional logic, we can derive a new partial binary operation $\odot$ as its conjunction. However, a logic is productive if it enables deductive reasoning for which a connective implication is necessary. Of course, such an implication can be introduced in many different ways. However, in substructural logics an implication $\rightarrow$ is considered to be appropriate if it is connected with conjunction via the so-called adjointness, which means that
\[
x\odot y\leq z\text{ if and only if }x\leq y\rightarrow z
\]
provided the conjunction is associative, commutative and monotonous. It means that if $a\odot b\leq c$ then $a$ is considered as a possible candidate for $b\rightarrow c$, i.e.\ $b\rightarrow c$ is ``as good as $a$'', in other words, $a$ is the lower estimation of $b\rightarrow c$.

Such an approach was investigated by the first author and R.~Hala\v s (\cite{CH}) for so-called lattice effect algebras, i.e.\ effect algebras that form a lattice with respect to the induced order. Here we define $x\rightarrow y:=x'+y$. Unfortunately, this adjointness holds here only if both the terms $x\odot y$ and $y\rightarrow z$ are defined which is just the case when $z\leq y$ and $x'\leq y$. This is an essential constraint and, moreover, it does not hold for effect algebras which do not form a lattice with respect to the induced order. Hence, there arise two questions:

(1) How to introduce the connective implication in lattice effect algebras in order to satisfy some kind of adjointness with everywhere defined terms.

(2) How to define implication in non-lattice effect algebras such that a certain kind adjointness is satisfied and all the terms occurring in this adjointness condition are everywhere defined.

The first question was answered by the authors in \cite{CL3} by defining $x\rightarrow y:=x'+(x\wedge y)$. We proved that then
\[
(x\vee y')\odot y\leq y\wedge z\text{ if and only if }x\vee y'\leq y\rightarrow z.
\]
This is a kind of adjointness which we called ``relative'' in \cite{CL1} because in contrast to the above mentioned adjointness, the element $y$ occurs also in the right-hand side of the first inequality and in the left-hand side of the second one.

Unfortunately, the connective implication cannot be introduced in this way for non-lattice effect algebras because the operation $\wedge$ need not be defined. But non-lattice effect algebras are more appropriate for the logic of quantum mechanics as pointed in the paper \cite{FB} and the monograph \cite{DP}. Hence the authors try another construction and define $x\rightarrow y:=x'+L(x,y)$, where $L(x,y)$ means the so-called lower cone of the elements $x$ and $y$. In this approach the result of $x\rightarrow y$ need not be an element of the effect algebra $\mathbf E$ in question but a subset of it. When considering a so-called monotonous effect algebra $\mathbf E$, i.e.\ an effect algebra satisfying the implication
\[
\text{if }A,B\leq x'\text{ and }L(A)\leq U(B)\text{ then }L(x+A)\leq U(x+B)
\]
(the details of this notation are explained below) then $\mathbf E$ satisfies the so-called unsharp adjointness
\[
L(U(x,y')\odot y)\leq UL(y,z)\text{ if and only if }LU(x,y')\leq U(y\rightarrow z),
\]
see \cite{CL3}. The term ``unsharp'' expresses the fact that the results of occuring terms are not necessary elements but subsets.

In the present paper we prove that if $\mathbf E$ is an effect algebra and no additional condition is assumed then for $\odot$ and this implication another version of adjointness holds, i.e.\ which shows that our concept of implication is sound.

Let us firstly recall several necessary concepts from ordered sets (posets) as well as the definition of effect algebra.

Let $(P,\leq)$ be a poset, $a,b\in P$ and $A,B\subseteq P$. We put
\begin{align*}
L(A) & :=\{x\in P\mid x\leq y\text{ for all }y\in A\}, \\
U(A) & :=\{x\in P\mid y\leq x\text{ for all }y\in A\}.
\end{align*}
Instead of $L(\{a\})$, $L(\{a,b\})$, $L(A\cup\{a\})$, $L(A\cup B)$ and $L(U(A))$ we simply write $L(a)$, $L(a,b)$, $L(A,a)$, $L(A,B)$ and $LU(A)$, respectively. Analogously, we proceed in similar cases.

\begin{definition}
An {\em effect algebra} is a partial algebra $\mathbf E=(E,+,{}',0,1)$ of type $(2,1,0,0)$ where $(E,{}',0,1)$ is an algebra and $+$ is a partial operation satisfying the following conditions for all $x,y,z\in E$:
\begin{enumerate}[{\rm(E1)}]
\item $x+y$ is defined if and only if so is $y+x$ and in this case $x+y=y+x$,
\item $(x+y)+z$ is defined if and only if so is $x+(y+z)$ and in this case $(x+y)+z=x+(y+z)$,
\item $x'$ is the unique $u\in E$ with $x+u=1$,
\item if $1+x$ is defined then $x=0$.
\end{enumerate}
On $E$ a binary relation $\leq$ can be defined by
\[
x\leq y\text{ if there exists some }z\in E\text{ with }x+z=y
\]
{\rm(}$x,y\in E${\rm)}. Then $(E,\leq,0,1)$ becomes a bounded poset and $\leq$ is called the {\em induced order} of $\mathbf E$. If $(E,\leq)$ is a lattice then $\mathbf E$ is called a lattice effect algebra.
\end{definition}

\begin{lemma}\label{lem1}
{\rm(}see {\rm\cite{DV}, \cite{FB})} If $(E,+,{}',0,1)$ is an effect algebra, $\leq$ its induced order and $a,b,c\in E$ then
\begin{enumerate}[{\rm(i)}]
\item $a''=a$,
\item $a\leq b$ implies $b'\leq a'$,
\item $a+b$ is defined if and only if $a\leq b'$,
\item if $a\leq b$ and $b+c$ is defined then $a+c$ is defined and $a+c\leq b+c$,
\item if $a\leq b$ then $a+(a+b')'=b$ and $(b'+(b'+a)')'=a$,
\item $a+0=0+a=a$,
\item $0'=1$ and $1'=0$.
\end{enumerate}
\end{lemma}

Let $(E,+,{}',0,1)$ be an effect algebra, $A,B\subseteq E$ and $a\in E$. Then $A\leq B$ means $x\leq y$ for all $x\in A$ and $y\in B$. Instead of $A\leq\{a\}$ and $\{a\}\leq A$ we write shortly $A\leq a$ and $a\leq A$, respectively. Moreover, we put $A':=\{x'\mid x\in A\}$ and in case $A\leq x'$ we put $x+A:=\{x+y\mid y\in A\}$. We can extend the operation $+$ also for subsets $A$ and $B$ of $E$ as follows: If $A\leq B'$ then $A+B:=\{x+y\mid x\in A,y\in B\}$.

\begin{lemma}\label{lem2}
Let $(E,+,{}',0,1)$ be an effect algebra and $a,b\in E$. Then
\begin{align*}
L(a,b) & =(a'+(a'+L(a,b))')', \\
U(a,b) & =a+(a+(U(a,b))')'.
\end{align*}
\end{lemma}

\begin{proof}
If $c\in L(a,b)$ and $d\in U(a,b)$ then because of $c\leq a\leq d$ we have by (v) of Lemma~\ref{lem1}
\begin{align*}
c & =(a'+(a'+c)')', \\
d & =a+(a+d')'.
\end{align*}
\end{proof}

A {\em partial commutative monoid} is a partial algebra $(A,\odot,1)$ of type $(2,0)$ where $\odot$ is a partial operation satisfying the following conditions for all $x,y,z\in A$:
\begin{itemize}
\item $(x\odot y)\odot z$ is defined if and only if so is $x\odot(y\odot z)$ and in this case $(x\odot y)\odot z=x\odot(y\odot z)$,
\item $x\odot1\approx1\odot x\approx x$,
\item $x\odot y$ is defined if and only if so is $y\odot x$ and in this case $x\odot y=y\odot x$.
\end{itemize}

Let $\mathbf E=(E,+,{}',0,1)$ be an effect algebra. For elements $x,y$ and subsets $A,B$ of $E$ we define
\begin{align*}
      x\cdot y & :=(x'+y')'\text{ if and only if }x'\leq y, \\
x\rightarrow y & :=x'+L(x,y), \\
x\rightarrow B & :=x'+L(x,B), \\
A\rightarrow y & :=A'+L(A,y), \\
A\rightarrow B & :=A'+L(A,B).
\end{align*}

We list several properties of the connective implication in effect algebras.

\begin{theorem}\label{th2}
Let $(E,+,{}',0,1)$ be an effect algebra with induced order $\leq$ and $a,b,c\in E$. Then
\begin{enumerate}[{\rm(i)}]
\item $a\rightarrow b\subseteq U(a')$,
\item if $a\leq b$ then $a\rightarrow b=U(a')$,
\item if $b\leq a$ then $a\rightarrow b=[a',a'+b]$,
\item $0\rightarrow b=\{1\}$,
\item $a\rightarrow0=\{a'\}$,
\item $1\rightarrow b=L(b)$,
\item $L(a\rightarrow b)=L(a')$,
\item $a\cdot(a\rightarrow b)=L(a,b)$,
\item if $b\leq c$ then $a\rightarrow b\subseteq a\rightarrow c$,
\item $a\rightarrow b=(a\cdot(L(a,b))')'=(a\cdot U(a',b'))'$,
\item $a\rightarrow b\leq U(a',c')$ if and only if $a\rightarrow c\leq U(a',b')$,
\item if $\mathbf E$ is a lattice effect algebra then $a\rightarrow(a\wedge b)=a\rightarrow b$.
\end{enumerate}
\end{theorem}

\begin{proof}
\
\begin{enumerate}[(i)]
\item We have $a\rightarrow b=a'+L(a,b)\subseteq U(a')$.
\item Assume $a\leq b$. According to (i), $a\rightarrow b\subseteq U(a')$. Conversely, if $c\in U(a')$ then $c=a'+(a'+c')'\in a'+L(a)$ since $a'\leq a'+c'$.
\item Assume $b\leq a$. Then $a\rightarrow b=a'+L(a,b)=a'+L(b)\subseteq[a',a'+b]$. Conversely, if $c\in[a',a'+b]$ then everyone of the following assertions implies the next one:
\begin{align*}
         c & \leq a'+b, \\
   (a'+b)' & \leq c', \\
a'+(a'+b)' & \leq a'+c', \\
  (a'+c')' & \leq(a'+(a'+b)')', \\
  (a'+c')' & \leq(b')', \\
  (a'+c')' & \leq b
\end{align*}
and hence $c=a'+(a'+c')'\in a'+L(b)=a\rightarrow b$.
\item According to (ii) we have $0\rightarrow b=U(0')=\{1\}$.
\item According to (iii) we have $a\rightarrow0=[a',a'+0]=\{a'\}$.
\item According to (iii) we have $1\rightarrow b=[1',1'+b]=L(b)$.
\item We have $L(a\rightarrow b)=L(a'+L(a,b))=L(a'+0)=L(a')$.
\item According to Lemma~\ref{lem2} we have
\begin{align*}
a\cdot(a\rightarrow b) & =(a'+(a\rightarrow b)')'=(a'+(a'+L(a,b))')'=(a'+(a'+(U(a',b'))')')'= \\
                       & =(U(a',b'))'=L(a,b).
\end{align*}
\item If $b\leq c$ then $a\rightarrow b=a'+L(a,b)\subseteq a'+L(a,c)=a\rightarrow c$.
\item We have $a\rightarrow b=a'+L(a,b)=(a\cdot(L(a,b))')'=(a\cdot U(a',b'))'$.
\item If $a\rightarrow b\leq U(a',c')$ then according to (viii) we have
\begin{align*}
a\rightarrow c & =a'+L(a,c)=a'+(U(a',c'))'\leq a'+(a\rightarrow b)'=(a\cdot(a\rightarrow b))'=(L(a,b))'= \\
               & =U(a',b').
\end{align*}
The converse implication follows by interchanging $b$ and $c$.
\item If $\mathbf E$ is a lattice effect algebra then
\[
a\rightarrow(a\wedge b)=a'+L(a,a\wedge b)=a'+L(a\wedge b)=a'+L(a,b)=a\rightarrow b.
\]
\end{enumerate}
\end{proof}

Further properties of the connective implication are collected in Theorem~\ref{th4} below. Now, some of the arguments of the implication operation are subsets of the corresponding effect algebra instead of elements.

\begin{theorem}\label{th4}
Let $\mathbf E=(E,+,{}',0,1)$ be an effect algebra and $a,b,c\in E$. Then
\begin{enumerate}[{\rm(i)}]
\item $(a\rightarrow0)\rightarrow0=\{a\}$ {\rm(}{\em double negation law}{\rm)},
\item $a\rightarrow(b\rightarrow c)=a\rightarrow b'$,
\item $a\rightarrow U(b)=a\rightarrow b$,
\item $U(a)\rightarrow b=U(a)\rightarrow U(b)=U(a',b')\rightarrow a'=L(a')+L(a,b)$,
\item $a\rightarrow L(a,b)=\{a'\}$,
\item $a\rightarrow U(a,b)=a'+L(a)$,
\item $U(a\rightarrow U(a,b))=\{1\}$.
\end{enumerate}
\end{theorem}

\begin{proof}
\
\begin{enumerate}[(i)]
\item According to (v) we have $(a\rightarrow0)\rightarrow0=\{a'\}\rightarrow0=\{a''\}+L(\{a'\},0)=\{a\}$.
\item We have $a\rightarrow(b\rightarrow c)=a'+L(a,b'+L(b,c))=a'+L(a,b')=a\rightarrow b'$.
\item We have $a\rightarrow U(b)=a'+L(a,U(b))=a'+L(a,b)=a\rightarrow b$.
\item We have
\begin{align*}
     U(a)\rightarrow b & =(U(a))'+L(U(a),b)=L(a')+L(a,b), \\
  U(a)\rightarrow U(b) & =(U(a))'+L(U(a),U(b))=L(a')+L(a,b), \\
U(a',b')\rightarrow a' & =(U(a',b'))'+L(U(a',b'),a')=L(a,b)+L(a')=L(a')+L(a,b).
\end{align*}
\item We have $a\rightarrow L(a,b)=a'+L(a,L(a,b))=\{a'\}$.
\item We have $a\rightarrow U(a,b)=a'+L(a,U(a,b))=a'+L(a)$.
\item According to (vi) we have $U(a\rightarrow U(a,b))=U(a'+L(a))=\{1\}$.
\end{enumerate}
\end{proof}

Now we are ready to define our main concept.

\begin{definition}
A {\em strict unsharp residuated poset} is an ordered seventuple $\mathbf C=(C,\leq,\odot,\rightarrow,{}',0,1)$ where $\rightarrow:C^2\rightarrow2^C$ and the following hold for all $x,y,z\in C$:
\begin{enumerate}[{\rm(C1)}]
\item $(C,\leq,{}',0,1)$ is a bounded poset with an antitone involution,
\item $(C,\odot,1)$ is a partial commutative monoid where $x\odot y$ is defined if and only if $x'\leq y$. Moreover, $z'\leq x\leq y$ implies $x\odot z\leq y\odot z$, and $x\leq y$ implies $x=y\odot(y\odot x')'$,
\item $U(x,y')\odot y\subseteq UL(y,z)$ if and only if $U(x,y')\subseteq U(y\rightarrow z)$,
\item $x\rightarrow0=\{x'\}$.
\end{enumerate}
The strict unsharp residuated poset $\mathbf C$ is called {\em divisible} if
\begin{enumerate}
\item[{\rm(C5)}] $x\odot(x\rightarrow y)=L(x,y)$
\end{enumerate}
for all $x,y\in C$.
\end{definition}

Condition {\rm(C3)} is called {\em unsharp adjointness}. Note that because of $y'\leq U(x,y')$ the expression $U(x,y')\odot y$ is everywhere defined. Since for two subsets $A,B$ of a poset $(P,\leq)$, $A\subseteq U(B)$ is equivalent to $A\geq B$, we have that unsharp adjointness (C3) is equivalent to the following {\em dual unsharp adjointness}:
\begin{enumerate}
\item[(C3')] $U(x,y')\odot y\geq L(y,z)$ if and only if $U(x,y')\geq y\rightarrow z$.
\end{enumerate}
The adjective ``unsharp'' means that the value of the implication is not an element but a subset of $C$. The adjective ``strict'' refers to the fact that the operation $x\odot y$ is defined strictly only in case $x'\leq y$.

The next theorem shows that every effect algebra can be organized into a strict unsharp residuated poset when the implication $\rightarrow$ is defined as above.

\begin{theorem}\label{th1}
Let $\mathbf E=(E,+,{}',0,1)$ be an effect algebra with induced order $\leq$ and put
\begin{align*}
      x\odot y & :=(x'+y')'\text{ if and only if }x'\leq y, \\
x\rightarrow y & :=x'+L(x,y)
\end{align*}
{\rm(}$x,y\in E${\rm)}. Then $\mathbb C(\mathbf E):=(E,\leq,\odot,\rightarrow,{}',0,1)$ is a divisible strict unsharp residuated poset.
\end{theorem}

\begin{proof}
Let $a,b,c\in E$.
\begin{enumerate}[(C1)]
\item This is obvious.
\item The first part of (C2) is clear. If $c'\leq a\leq b$ then
\[
a\odot c=(a'+c')'\leq(b'+c')'=b\odot c.
\]
If $a\leq b$ then $b'\leq a'$ and hence
\[
a=(a')'=(b'+(b'+a)')'=b\odot(b\odot a')'.
\]
\item According to (x) and (xi) of Theorem~\ref{th2} the following are equivalent:
\begin{align*}
    U(a,b')\odot b & \subseteq UL(b,c), \\
    b\odot U(b',a) & \subseteq UL(b,c), \\
(b\rightarrow a')' & \subseteq UL(b,c), \\
            L(b,c) & \leq(b\rightarrow a')', \\
   b\rightarrow a' & \leq U(b',c'), \\
    b\rightarrow c & \leq U(b',a), \\
           U(a,b') & \subseteq U(b\rightarrow c).
\end{align*}
\item This follows from (v) of Theorem~\ref{th2}.
\item This follows from (viii) of Theorem~\ref{th2}.
\end{enumerate}
\end{proof}

\begin{theorem}
Let $(E,+,{}',0,1)$ be an effect algebra with induced order $\leq$ and put
\begin{align*}
  x\odot y= xy & :=(x'+y')'\text{ if and only if }x'\leq y, \\
x\rightarrow y & :=x'+L(x,y)
\end{align*}
{\rm(}$x,y\in E${\rm)} then {\rm(C3)} and {\rm(xi)} of Theorem~\ref{th2} are equivalent.
\end{theorem}

\begin{proof}
Let $a,b,c\in E$. If (C3) holds then the following are equivalent:
\begin{align*}
 a\rightarrow b & \leq U(a',c'), \\
       U(c',a') & \geq a\rightarrow b, \\
       U(c',a') & \subseteq U(a\rightarrow b), \\
U(c',a')\odot a & \subseteq UL(a,b), \\
U(c',a')\odot a & \geq L(a,b), \\
      L(c,a)+a' & \leq U(a',b'), \\
      a'+L(a,c) & \leq U(a',b'), \\
 a\rightarrow c & \leq U(a',b'),
\end{align*}
i.e.\ (xi) of Theorem~\ref{th2} holds. If, conversely, (xi) of Theorem~\ref{th2} holds then the following are equivalent:
\begin{align*}
 U(a,b')\odot b & \subseteq UL(b,c), \\
 U(a,b')\odot b & \geq L(b,c), \\
     L(a',b)+b' & \leq U(b',c'), \\
     b'+L(b,a') & \leq U(b',c'), \\
b\rightarrow a' & \leq U(b',c'), \\
 b\rightarrow c & \leq U(b',a), \\
        U(a,b') & \geq b\rightarrow c, \\
        U(a,b') & \subseteq U(b\rightarrow c),
\end{align*}
i.e.\ (C3) holds.
\end{proof}

Also conversely, every strict unsharp residuated poset $\mathbf C$ can be considered as an effect algebra, in fact we need only the reduct of $\mathbf C$ in the similarity type $\{\leq,\odot,{}',0,1\}$, see the following theorem.

\begin{theorem}
Let $\mathbf C=(C,\leq,\odot,\rightarrow,{}',0,1)$ be a strict unsharp residuated poset and put
\[
x+y:=(x'\odot y')'\text{ if and only if }x\leq y'
\]
{\rm(}$x,y\in C${\rm)}. Then $\mathbb E(\mathbf C):=(C,+,{}',0,1)$ is an effect algebra whose induced order coincides with the order in $\mathbf C$.
\end{theorem}

\begin{proof}
Let $a,b\in C$. Obviously, (E1), (E2) and (E4) hold.
\begin{enumerate}
\item[(E3)] Since $0\leq a$ we have
\[
a+a'=a'+a=(a\odot a')'=(a\odot(a\odot1)')'=(a\odot(a\odot0')')'=0'=1
\]
according to (C2). Conversely, if $a+b=1$ then $a\leq b'$ and hence $b\leq a'$ which implies
\[
b=a'\odot(a'\odot b')'=a'\odot(a+b)=a'\odot1=a'
\]
again according to (C2).
\end{enumerate}
\end{proof}

Every effect algebra can be reconstructed from its assigned strict unsharp residuated poset as the following theorem shows.

\begin{theorem}
Let $\mathbf E$ be an effect algebra. Then $\mathbb E(\mathbb C(\mathbf E))=\mathbf E$.
\end{theorem}

\begin{proof}
Let
\begin{align*}
                      \mathbf E & =(E,+,{}',0,1)\text{ with induced order }\leq, \\
           \mathbb C(\mathbf E) & =(E,\leq,\odot,\rightarrow,{}',0,1), \\
\mathbb E(\mathbb C(\mathbf E)) & =(E,\oplus,{}',0,1)
\end{align*}
and $a,b\in E$. Then the following are equivalent:
\begin{align*}
& a\oplus b\text{ is defined}, \\
& a\leq b'\text{ in }\mathbb C(\mathbf E), \\
& a\leq b'\text{ in }\mathbf E, \\
& a+b\text{ is defined}
\end{align*}
and in this case
\[
a\oplus b=(a'\odot b')'=(a''+b'')''=a+b.
\]
\end{proof}

As mentioned in the introduction, the connective implication enables deductive reasoning in propositional logic. It is usually ruled by means of Modus Ponens. In classical logic Modus Ponens says that if both the proposition $x$ and the implication $x\rightarrow y$ are true then also the assertion $y$ is true. In non-classical logics, Modus Ponens expresses the fact that the true-value of $y$ cannot be less than true-values of $x$ and of $x\rightarrow y$. In order to algebraize this situation often so-called deductive systems are used. If $A$ denotes the set of true-values of a certain logic then a non-empty subset $D$ of $A$ is called a deductive system of $A$ if $x,x\rightarrow y\in D$ imply $y\in D$. In case $D=\{1\}$ we obtain the classical Modus Ponens.

For our unsharp implication we modify the definition of a deductive system as follows.

\begin{definition}
Let $\mathbf E=(E,+,{}',0,1)$ be an effect algebra. A subset $D$ of $E$ is called a {\em deductive system} of $\mathbf E$ if for $x,y\in E$ we have
\begin{itemize}
\item $1\in D$,
\item if $x\in D$ and $x\rightarrow y\subseteq D$ then $y\in D$.
\end{itemize}
Let $\Ded(\mathbf E)$ denote the set of all deductive systems of $\mathbf E$.
\end{definition}

\begin{theorem}\label{th3}
Let $\mathbf E=(E,+,{}',0,1)$ be an effect algebra. Then $E\in\Ded(\mathbf E)$. Moreover, a proper subset $D$ of $E$ containing $1$ is a deductive system of $\mathbf E$ if and only if $D\cap D'=\emptyset$.
\end{theorem}

\begin{proof}
Let $D$ be a proper subset of $E$ containing the element $1$. First assume $D\in\Ded(\mathbf E)$. Suppose $D\cap D'\neq\emptyset$. Then there exists some $a\in D\cap D'$. Hence there exists some $b\in D$ with $b'=a$. Since $b\in D$ and $b\rightarrow0=\{b'\}\subseteq D$ according to Theorem~\ref{th2}, we have $0\in D$. If $c\in E$ then because of $0\in D$ and $0\rightarrow c=\{1\}\subseteq D$ according to Theorem~\ref{th2}, we have $c\in D$. This shows $D=E$, a contradiction. Hence $D\cap D'=\emptyset$. Conversely, assume $D\cap D'=\emptyset$. Assume $d\in D$, $e\in E$ and $d\rightarrow e\subseteq D$. Then $d'=d'+0\in d'+L(d,e)=d\rightarrow e\subseteq D$ and hence $d\in D\cap D'$ contradicting $D\cap D'=\emptyset$. Hence the situation $d\in D$, $e\in E$ and $d\rightarrow e\subseteq D$ is impossible which shows $D\in\Ded(\mathbf E)$ completing the proof of the theorem.
\end{proof}

From this proof we see that if $D$ is a proper subset of $E$ containing the element $1$ then $D\in\Ded\mathbf E$ if and only if there do not exist $a\in D$ and $b\in E$ with $a\rightarrow b\subseteq D$.

\begin{theorem}
Let $\mathbf E=(E,+,{}',0,1)$ be an effect algebra and $M\subseteq E$. Then
\begin{enumerate}[{\rm(i)}]
\item $\BDed(\mathbf E):=(\Ded(\mathbf E),\subseteq)$ is a complete lattice with the smallest element $\{1\}$ and the greatest element $E$.
\item If there exists some $x\in E\setminus\{0,1\}$ with $x'\neq x$ then $\BDed(\mathbf E)$ is atomic and its atoms are exactly the sets of the form $\{1,x\}$ with $x\in E\setminus\{0,1\}$ and $x'\neq x$.
\item The deductive system of $\mathbf E$ generated by $M$ coincides with $M\cup\{1\}$ if both $M\cap M'=\emptyset$ and $0\notin M$ and with $E$ otherwise.
\end{enumerate}
\end{theorem}

\begin{proof}
Let $a\in E$.
\begin{enumerate}[(i)]
\item This is clear.
\item Assume there exists some $b\in E\setminus\{0,1\}$ with $b'\neq b$. If $a\neq0,1$ and $a'\neq a$ then $\{1,a\}\cap\{1,a\}'=\emptyset$ and hence $\{1,a\}\in\Ded(\mathbf E)$ according to Theorem~\ref{th3}. Now let $D\in\Ded(\mathbf E)$ with $D\neq\{1\}$. If $D=E$ then $\{1,b\}\subseteq D$. Now assume $D\neq E$. If $D=\{0,1\}$ then we would have $D\cap D'\neq\emptyset$ and hence $D=E$ according to Theorem~\ref{th3} contradicting $|E|>2$. Therefore there exists some $c\in D\setminus\{0,1\}$. Now $c'=c$ would imply $c\in D\cap D'$ contradicting $D\in\Ded(\mathbf E)$ according to Theorem~\ref{th3}. Hence $c'\neq c$ and $\{1,c\}\subseteq D$ completing the proof.
\item Without loss of generality assume $|E|>1$. Then $0\neq1$. If $M\cap M'=\emptyset$ and $0\notin M$ then either $M\cup\{1\}=E\in\Ded(\mathbf E)$ or we have both $M\cup\{1\}\neq E$ and $(M\cup\{1\})\cap(M\cup\{1\})'=(M\cup\{1\})\cap(M'\cup\{0\})=\emptyset$ and hence $M\cup\{1\}\in\Ded(\mathbf E)$ according to Theorem~\ref{th3}. Now let $F$ denote the deductive system of $\mathbf E$ generated by $M$. If $M\cap M'\neq\emptyset$ then $F\cap F'\neq\emptyset$ and hence $F=E$ according to Theorem~\ref{th3}. If $0\in M$ then $0,1\in F$ and hence $F\cap F'\neq\emptyset$ which again implies $F=E$.
\end{enumerate}
\end{proof}

Let $\mathbf E=(E,+,{}',0,1)$ be an effect algebra and $a,b\in E$. Then we can consider the set $a\rightarrow b$ as the set of truth-values of this implication. Since $a\rightarrow b=a'+L(a,b)$, the set $a\rightarrow b$ has the least element $a'$. However, the deductive power of implication is given by its highest values because these reveal the reliable truth-values of this deduction. We say that a logic equipped with connectives implication $\rightarrow$ and negation $'$ satisfies the {\em contraposition law} if
\[
x\rightarrow y\approx y'\rightarrow x'.
\]
For our case of unsharp implication, we modify this law in the form
\[
U(x\rightarrow y)\approx U(y'\rightarrow x').
\]
and call it the {\em unsharp contraposition law}.

\begin{proposition}\label{prop1}
Let $\mathbf E=(E,+,{}',0,1)$ be an effect algebra and $a,b\in E$. Then the unsharp contraposition law holds for $a$ and $b$ provided they are comparable.
\end{proposition}

\begin{proof}
If $a\leq b$ then
\[
U(a\rightarrow b)=U(U(a'))=\{1\}=U(U(b''))=U(b'\rightarrow a')
\]
according to (ii) of Theorem~\ref{th2}, and if $b\leq a$ then
\begin{align*}
U(a\rightarrow b) & =U([a',a'+b])=U(a'+b)=U(b+a')=U(b''+a')=U([b'',b''+a'])= \\
                  & =U(b'\rightarrow a').
\end{align*}
according to (iii) of Theorem~\ref{th2}.
\end{proof}

The following example shows that the unsharp contraposition law need not hold for non-comparable elements provided the effect algebra is not lattice-ordered.

\begin{example}
Let $E$ denote the $9$-element set $\{0,a,b,c,d,e,f,g,1\}$ and define $+$ and $'$ as follows:
\[
\begin{array}{c|ccccccccc}
+ & 0 & a & b & c & d & e & f & g & 1 \\
\hline
0 & 0 & a & b & c & d & e & f & g & 1 \\
a & a & - & e & f & - & - & - & 1 & - \\
b & b & e & d & g & f & - & 1 & - & - \\
c & c & f & g & - & - & 1 & - & - & - \\
d & d & - & f & - & 1 & - & - & - & - \\
e & e & - & - & 1 & - & - & - & - & - \\
f & f & - & 1 & - & - & - & - & - & - \\
g & g & 1 & - & - & - & - & - & - & - \\
1 & 1 & - & - & - & - & - & - & - & -
\end{array}
\quad
\begin{array}{c|c}
x & x' \\
\hline
0 & 1 \\
a & g \\
b & f \\
c & e \\
d & d \\
e & c \\
f & b \\
g & a \\
1 & 0
\end{array}
\]
Then $(E,+,{}',0,1)$ is an effect algebra that is not lattice-ordered as the Hasse diagram of $(E,\leq)$ depicted in Fig.~1 shows.

\vspace*{-2mm}

\[
\setlength{\unitlength}{7mm}
\begin{picture}(6,9)
\put(3,2){\circle*{.3}}
\put(1,4){\circle*{.3}}
\put(3,4){\circle*{.3}}
\put(5,4){\circle*{.3}}
\put(3,5){\circle*{.3}}
\put(1,6){\circle*{.3}}
\put(3,6){\circle*{.3}}
\put(5,6){\circle*{.3}}
\put(3,8){\circle*{.3}}
\put(3,2){\line(-1,1)2}
\put(3,2){\line(0,1)6}
\put(3,2){\line(1,1)2}
\put(1,6){\line(0,-1)2}
\put(1,6){\line(1,-1)2}
\put(1,6){\line(1,1)2}
\put(5,6){\line(0,-1)2}
\put(5,6){\line(-1,-1)2}
\put(5,6){\line(-1,1)2}
\put(3,6){\line(-1,-1)2}
\put(3,6){\line(1,-1)2}
\put(2.875,1.25){$0$}
\put(.35,3.85){$a$}
\put(3.4,3.85){$b$}
\put(5.4,3.85){$c$}
\put(3.4,4.85){$d$}
\put(.35,5.85){$e$}
\put(3.4,5.85){$f$}
\put(5.4,5.85){$g$}
\put(2.85,8.4){$1$}
\put(2.2,.3){{\rm Fig.~1}}
\end{picture}
\]

\vspace*{-3mm}

The operation table for $\rightarrow$ looks as follows:
\[
\begin{array}{c|c|c|c|c|c|c|c|c|c}
\rightarrow &   0   &    a    &    b    &    c    &     d     &      e      &        f        &      g      &        1 \\
\hline
      0     & \{1\} &  \{1\}  &  \{1\}  &  \{1\}  &   \{1\}   &    \{1\}    &      \{1\}      &    \{1\}    &      \{1\} \\
\hline
			a     & \{g\} & \{g,1\} &  \{g\}  &  \{g\}  &   \{g\}   &   \{g,1\}   &     \{g,1\}     &    \{g\}    &     \{g,1\} \\
\hline
			b     & \{f\}	&  \{f\}	& \{f,1\}	&  \{f\}	&  \{f,1\}	&   \{f,1\}	  &     \{f,1\}	    &   \{f,1\}	  &     \{f,1\}	\\
\hline
			c     & \{e\} &  \{e\}  &  \{e\}  & \{e,1\} &   \{e\}   &    \{e\}    &     \{e,1\}     &   \{e,1\}   &     \{e,1\} \\
\hline
			d     & \{d\} &  \{d\}  & \{d,f\} &  \{d\}  & \{d,f,1\} &   \{d,f\}   &    \{d,f,1\}    &   \{d,f\}   &    \{d,f,1\} \\
\hline
			e     & \{c\} & \{c,f\} & \{c,g\} &  \{c\}  &  \{c,g\}  & \{c,f,g,1\} &    \{c,f,g\}    &   \{c,g\}   &   \{c,f,g,1\} \\
\hline
			f     & \{b\} & \{b,e\} & \{b,d\} & \{b,g\} & \{b,d,f\} &  \{b,d,e\}  & \{b,d,e,f,g,1\} &  \{b,d,g\}  & \{b,d,e,f,g,1\} \\
\hline
			g     & \{a\} &  \{a\}  & \{a,e\} & \{a,f\} &  \{a,e\}  &   \{a,e\}   &    \{a,e,f\}    & \{a,e,f,1\} &   \{a,e,f,1\} \\
\hline
			1     & \{0\} & \{0,a\} & \{0,b\} & \{0,c\} & \{0,b,d\} & \{0,a,b,e\} & \{0,a,b,c,d,f\} & \{0,b,c,g\} & E
\end{array}
\]
One can see that $a$ and $d$ are incomparable and
\[
U(a\rightarrow d)=U(g)=\{g,1\}\neq\{f,1\}=U(d,f)=U(d\rightarrow g)=U(d'\rightarrow a').
\]
On the contrary, if $x,y\in E$ are comparable then the $x\rightarrow y$ and $y'\rightarrow x'$ need not coincide, but their upper cones are equal in accordance with Proposition~\ref{prop1}. For example, $a\leq e$ and
\[
e\rightarrow a=\{c,f\}\neq\{a,f\}=g\rightarrow c=a'\rightarrow e',
\]
but
\[
U(e\rightarrow a)=U(c,f)=\{f,1\}=U(a,f)=U(a'\rightarrow e').
\]
\end{example}

For lattice effect algebras, every implication $a\rightarrow b$ can be expressed equivalently as $a\rightarrow(a\wedge b)$, see (xii) of Theorem~\ref{th2}, i.e.\ with comparable elements $a$ and $a\wedge b$. Hence, for lattice effect algebras we have this variant of the contraposition law
\[
U(x\rightarrow y)\approx U(x\rightarrow(x\wedge y)\approx U((x\wedge y)'\rightarrow x').
\]
However, we can state the following result.

\begin{proposition}\label{prop2}
Let $\mathbf E=(E,+,{}',0,1)$ be a lattice effect algebra. Then $\mathbf E$ satisfies the unsharp contraposition law if and only if it satisfies the identity
\begin{equation}\label{equ1}
x'+(x\wedge y)\approx y+(x'\wedge y').
\end{equation}
\end{proposition}

\begin{proof}
If $a,b\in E$ then
\begin{align*}
  U(a\rightarrow b) & =U(a'+L(a,b))=U(a'+L(a\wedge b))=U(a'+(a\wedge b)), \\
U(b'\rightarrow a') & =U(b+L(b',a'))=U(b+L(a'\wedge b'))=U(b+(a'\wedge b')).
\end{align*}
\end{proof}

The following example shows that lattice effect algebras need not satisfy identity (\ref{equ1}).

\begin{example}
Let $E$ denote the $6$-element set $\{0,a,a',b,b',1\}$ and define $+$ and $'$ as follows:
\[
\begin{array}{c|cccccc}
+  & 0 & a & a' & b & b' & 1 \\
\hline
0  & 0  & a & a' & b & b' & 1 \\
a  & a  & - & 1  & - & -  & - \\
a' & a' & 1 & -  & - & -  & - \\
b  & b  & - & -  & - & 1  & - \\
b' & b' & - & -  & 1 & -  & - \\
1  & 1  & - & -  & - & -  & -
\end{array}
\quad
\begin{array}{c|c}
x & x' \\
\hline
0  & 1 \\
a  & a' \\
a' & a \\
b  & b' \\
b' & b \\
1  & 0
\end{array}
\]
Then $\mathbf E:=(E,+,{}',0,1)$ is a lattice effect algebra as the Hasse diagram of $(E,\leq)$ depicted in Fig.~2 shows

\vspace*{-7mm}

\[
\setlength{\unitlength}{7mm}
\begin{picture}(8,7)
\put(4,2){\circle*{.3}}
\put(1,4){\circle*{.3}}
\put(3,4){\circle*{.3}}
\put(5,4){\circle*{.3}}
\put(7,4){\circle*{.3}}
\put(4,6){\circle*{.3}}
\put(4,2){\line(-3,2)3}
\put(4,2){\line(-1,2)1}
\put(4,2){\line(1,2)1}
\put(4,2){\line(3,2)3}
\put(4,6){\line(-3,-2)3}
\put(4,6){\line(-1,-2)1}
\put(4,6){\line(1,-2)1}
\put(4,6){\line(3,-2)3}
\put(3.875,1.25){$0$}
\put(.4,3.85){$a$}
\put(2.35,3.85){$a'$}
\put(5.3,3.85){$b$}
\put(7.3,3.85){$b'$}
\put(3.85,6.4){$1$}
\put(3.2,.3){{\rm Fig.~2}}
\end{picture}
\]

\vspace*{-3mm}

and we have
\[
a'+(a\wedge b)=a'+0=a'\neq b=b+0=b+(a'\wedge b').
\]
Thus, according to Proposition~\ref{prop2}, $\mathbf E$ does not satisfy the unsharp contraposition law.
\end{example}

Note that if $(B,\vee,\wedge,{}',0,1)$ is a Boolean algebra and a binary operation $+$ on $B$ is defined by
\[
x+y:=x\vee y\text{ if and only if }x\leq y'
\]
then $(B,+,{}',0,1)$ is a lattice effect algebra satisfying the unsharp contraposition law according to Proposition~\ref{prop2} since we have
\begin{align*}
x'+(x\wedge y) & \approx x'\vee(x\wedge y)\approx(x'\vee x)\wedge(x'\vee y)\approx1\wedge(x'\vee y)\approx x'\vee y\approx y\vee x'\approx \\
               & \approx(y\vee x')\wedge1\approx(y\vee x')\wedge(y\vee y')\approx y\vee(x'\wedge y')\approx y+(x'\wedge y').
\end{align*}

Authors' addresses:

Ivan Chajda \\
Palack\'y University Olomouc \\
Faculty of Science \\
Department of Algebra and Geometry \\
17.\ listopadu 12 \\
771 46 Olomouc \\
Czech Republic \\
ivan.chajda@upol.cz

Helmut L\"anger \\
TU Wien \\
Faculty of Mathematics and Geoinformation \\
Institute of Discrete Mathematics and Geometry \\
Wiedner Hauptstra\ss e 8-10 \\
1040 Vienna \\
Austria, and \\
Palack\'y University Olomouc \\
Faculty of Science \\
Department of Algebra and Geometry \\
17.\ listopadu 12 \\
771 46 Olomouc \\
Czech Republic \\
helmut.laenger@tuwien.ac.at
\end{document}